\documentclass[10pt]{article}
\usepackage{amsmath,mathdots}
\usepackage{amsfonts}
\usepackage{amssymb}
\usepackage{marginnote}
\usepackage[linesnumbered,ruled,vlined]{algorithm2e}

\usepackage{amsthm,verbatim}
\newtheorem{theorem}{Theorem}

\newtheorem{proposition}[theorem]{Proposition}
\newtheorem{corollary}[theorem]{Corollary}

\newtheorem{num_example}{Example}

\usepackage{todonotes}

\newcommand{\stella}{T} 
\renewcommand{\phi}{\varphi}
\newcommand{\wt}{\widetilde}
\newcommand{\wh}{\widehat}

\begin{document}

\title{Palindromic linearization and numerical solution of nonsymmetric algebraic $\stella$-Riccati equations\footnote{Version of \today. This work was partly supported by INdAM through GNCS projects.}}

\author{Peter Benner\thanks{Department Computational Methods in Systems and Control Theory (CSC), Max Planck Institute for Dynamics
of Complex Technical Systems, Magdeburg, Germany, %
  {\tt benner@mpi-magdeburg.mpg.de}} \and Bruno Iannazzo\thanks{Universit\`a degli Studi di Perugia, 
  Italy, {\tt bruno.iannazzo@unipg.it}} \and Beatrice Meini\thanks{Universit\`a di Pisa, Italy, {\tt beatrice.meini@unipi.it}} \and Davide Palitta\thanks{Universita di Bologna, Centro $AM^2$, Dipartimento di Matematica, Piazza di Porta S.
Donato 5, 40127 Bologna (Italy),   {\tt davide.palitta@unibo.it}} }

\date{}

\maketitle





\begin{abstract}
We identify a relationship between the solutions of a nonsymmetric algebraic $\stella$-Riccati equation ($\stella$-NARE) and the deflating subspaces of a palindromic matrix pencil, obtained by arranging the coefficients of the $\stella$-NARE.
The interplay between $\stella$-NARE and palindromic pencils allows one to derive both theoretical properties of the solutions of the equation, and new methods for its numerical solution.
In particular, we propose methods based on the (palindromic) QZ algorithm and the doubling algorithm, whose effectiveness is demonstrated by several numerical tests.\\[2ex]
\end{abstract}

\section{Introduction}

We consider the Nonsymmetric Algebraic $\stella$-Riccati equation ($\stella$-NARE)
\begin{equation}\label{eq:starnare}
DX+X^\stella A-X^\stella BX+C=0,
\end{equation}
where $A,B,C,D\in\mathbb{R}^{n\times n}$ are coefficient matrices and the $n\times n$ matrix $X$ is the unknown, while the superscript $\stella$ denotes transposition.

Equation \eqref{eq:starnare} has been considered in \cite{bp20}, with applications to Dynamic Stochastic General Equilibrium models. 
In the same paper, existence results are given in terms of nonnegativity properties of the coefficients, together with an analysis of the convergence of Netwon's method to the required solution. Recently, the problem of the existence of solutions, when $A=D=0$ and $C$ is symmetric, has been studied in \cite{borobia}.

The nonsymmetric $\stella$-Riccati equation takes the name from the Nonsymmetric Algebraic Riccati Equation (NARE)
\begin{equation}\label{eq:nare}
	DX+XA-XBX+C=0,
\end{equation}
that has received great attention in the literature in the last decades (see for instance the book \cite{bim:book}), and from the $\stella$ operator applied to the unknown $X$ in~\eqref{eq:starnare}, when this premultiplies a matrix coefficient.
Indeed, recently, there has been a great interest in the $\stella$ counterpart of classical linear matrix equations. For instance, look at the Sylvester equation $AX+XB=C$, whose $\stella$ counterpart $AX^\stella+X B=C$ has been the subject of a number of both computational and theoretical papers; see, e.g., \cite{bk}, \cite{dd11}, \cite{di16}, \cite{dipr18}, \cite{dipr19}, \cite{dgks16}, \cite{jp}.

It is well-known that there is a strict connection between equation \eqref{eq:nare} and the matrix
\begin{equation}\label{eq:H}
	H=\begin{bmatrix} A & -B\\-C & -D\end{bmatrix}.
\end{equation}
More specifically, $X$ is a solution to \eqref{eq:nare} if and only if there exists an $n$-dimensional invariant subspace of $H$ spanned by the columns of $\begin{bmatrix} I \\ X\end{bmatrix}$, where $I$ is the identity matrix (compare \cite[Thm.~2.1]{bim:book}).
This property reduces a nonlinear equation to an eigenvalue problem, that can be associated with the linear matrix polynomial $H-zI$. For this reason, perhaps with some abuse, we say that $H-zI$ is a linearization of the matrix equation \eqref{eq:nare}. (Note, however, the similarity with the linearization procedure for matrix polynomials.)
 
Our main contribution is the introduction of a linearization for the \mbox{$\stella$-NARE}~\eqref{eq:starnare}. More specifically, we define the $\stella$-palindromic pencil $\varphi(z)=M+zM^\stella$, where
\[
M=\begin{bmatrix}
C & D\\A & -B\end{bmatrix},
\]
and $A$, $B$, $C$, and $D$ are the matrix coefficients in \eqref{eq:starnare}.
We show that, if the pencil $\varphi(z)$ is regular and if $X$ is a solution to \eqref{eq:starnare}, then the columns of $\begin{bmatrix} I\\X\end{bmatrix}$ span a deflating subspace of $\varphi(z)$ and also a partial converse result holds. 
The precise statement is given in Theorem \ref{thm:main}. 
We say that the pencil $\varphi(z)$ is a linearization of the $T$-NARE \eqref{eq:starnare}.
We also relate the solutions of the $T$-NARE with the solutions of a suitable Discrete-time Algebraic Riccati equation. Moreover, we show existence properties of the solutions under nonnegativity properties of the matrix coefficients, by weakening the assumptions of \cite{bp20}.

The linearization of the $T$-NARE, besides being \emph{per se} interesting, is exploited to relate the solutions to~\eqref{eq:starnare} with the solution of certain discrete-time algebraic Riccati equations. More practically, it opens the way to compute the solution of a $\stella$-NARE by relying on invariant subspaces algorithms, such as the QZ and the Doubling Algorithm (DA), that in our tests are shown to be more efficient than Newton's method, the reference algorithm in \cite{bp20}.

Moreover, it is possible to exploit the palindromic structure of the linearization by applying the palindromic QZ algorithm, a structured variant of the QZ, that we endow with an ordering procedure and that is shown to be superior, in terms of forward error, in some difficult problems.

The paper is organized as follows: Section \ref{sec:prel} provides some preliminary material on matrix pencils, in Section \ref{sec:line} we present the interplay between $\stella$-NAREs and $T$-palindromic pencils, while Section \ref{sec:posi} is devoted to the case of matrix coefficients with nonnegativity properties. In Section \ref{sec:algs} we describe the new algorithms for the $\stella$-NARE and some tests are performed in Section \ref{sec:test}. The final section draws some conclusions.


\section{Preliminaries}\label{sec:prel}

We recall some properties of matrix pencils that will be used in the sequel. For more details on this subject we refer the reader to \cite{glr}.
 
A matrix pencil $p(z)=A+zB$, with $A,B\in\mathbb C^{n\times n}$ is said to be regular if $\det(p(z))$ is not identically 0. The finite eigenvalues of a regular pencil $p(z)$ are the zeros of $\det(p(z))$, while infinity is an eigenvalue of $p(z)$ with multiplicity $d\ge 1$ if $\det(p(z))$ has degree $n-d$. Notice that, for a regular pencil, $0$ is an eigenvalue if and only if $A$ is singular, while $\infty$ is an eigenvalue if and only if $B$ is singular.

A $k$-dimensional subspace $\mathcal V\subset \mathbb C^n$ is a deflating subspace of the regular matrix pencil $p(z)=A+zB$ if there exists a $k$-dimensional subspace $\mathcal W\subset \mathbb C^n$ such that $A\mathcal V\subset \mathcal W$ and $B\mathcal V\subset \mathcal W$. If the columns of $V\in\mathbb C^{n\times k}$ and $W\in\mathbb C^{n\times k}$ span the subspaces $\mathcal V$ and $\mathcal W$, respectively, then there exist $A',B'\in\mathbb C^{k\times k}$ such that $AV=WA'$ and $BV=WB'$. The eigenvalues of the $k\times k$ pencil $A'+zB'$ are a subset of the eigenvalues of $p(z)$ and are said to be the eigenvalues of $p(z)$ associated with the deflating subspace $\mathcal V$.

A deflating subspace $\mathcal V$ is said to be a graph deflating subspace if, for any matrix $V\in\mathbb C^{n\times k}$ whose columns span $\mathcal V$, the leading $k\times k$ submatrix of $V$ is invertible. In other words, $\mathcal V$ admits a basis made of the columns of the matrix $\left[\begin{smallmatrix} I\\X\end{smallmatrix}\right]$, for some matrix $X$.

A pencil $p(z)=A+zB$ is said to be $\stella$-palindromic if $B=A^\stella$, in this case $p(z)^\stella=zp(1/z)$, that is, if $\lambda$ is an eigenvalue of $p(z)$, then $1/\lambda$ is an eigenvalue of $p(z)$ as well (this holds also for $0$ and $\infty$, with the conventions $1/0=\infty$ and $1/\infty=0$). 


A reciprocal-free set $\mathcal S\subset \mathbb C\cup\{\infty\}$ is a set such that if $\lambda \in\mathcal S$ then $1/\lambda\not\in\mathcal S$.

Given  $x\in\mathbb{C}^n$ and a matrix $S\in\mathbb{C}^{n\times n}$, the vector $x$ is  said to be $S$-isotropic if $x^\stella S x = 0$.  More generally, if $\mathcal V$ 
is the subspace spanned by the columns of a full rank matrix $V\in\mathbb{C}^{n\times k}$, then the subspace $\mathcal V$
and the matrix $V$ are said to be $S$-isotropic if $V^\stella  S V = 0$.



\section{A linearization for nonsymmetric algebraic $\stella$-Riccati equations}\label{sec:line}

In this section we introduce a linearization of the $T$-NARE \eqref{eq:starnare}, which allows us to show theoretical properties of the solutions and to relate
the $\stella$-NARE to a suitable discrete-time algebraic Riccati equation.

We associate with the $\stella$-NARE \eqref{eq:starnare} the 
 $2n\times 2n$ $T$-palindromic pencil 
$	
$
\begin{equation}\label{eq:phi}
\phi(z)=M+zM^\stella,~~M=\left[\begin{matrix}
	C & D\\ A & -B
	\end{matrix}\right].
\end{equation}
The solutions of the $\stella$-NARE~\eqref{eq:starnare} are related to the deflating subspaces of the pencil $\phi(z)$ in \eqref{eq:phi}. In particular, the following result provides a necessary and sufficient condition for a matrix $X$ to be a solution of \eqref{eq:starnare}, in terms of the properties of the deflating subspaces of $\phi(z)$.

\begin{theorem}\label{thm:main}
Assume that the pencil $\varphi(z)$ in \eqref{eq:phi} is regular. If the matrix $X$ is a solution to \eqref{eq:starnare}, then
\begin{equation}\label{eq:defl}
	\varphi(z)\begin{bmatrix} I\\X\end{bmatrix} = \begin{bmatrix} -X^\stella\\ I\end{bmatrix} \alpha(z),
\end{equation}
where
\begin{equation}\label{eq:alpha}
\alpha(z):=A-BX+z(D^\stella-B^\stella X),
\end{equation}
i.e., 
the columns of $\left[\begin{smallmatrix} I\\X\end{smallmatrix}\right]$ span a deflating subspace of the pencil $\varphi(z)$, associated with the 
eigenvalues of the $n\times n$ pencil $\alpha(z)$.

Conversely, if $X$ is an $n\times n$ matrix such that the columns of the matrix $\left[\begin{smallmatrix} I\\X\end{smallmatrix}\right]$ span a deflating subspace
of $\phi(z)$ corresponding to a reciprocal-free set of eigenvalues, then $X$  is a solution to  \eqref{eq:starnare}.
\end{theorem}
\begin{proof}
If $X$ is a solution to \eqref{eq:starnare}, then we have 
\begin{equation}\label{eq:1}
	M\begin{bmatrix}
	I \\ X
	\end{bmatrix}
		=
   \begin{bmatrix}
	-X^\stella \\ I
	\end{bmatrix}(A-B X).
\end{equation}
Moreover, $X$ solves also the equation obtained by applying the $\stella$ operator to all the summands of equation \eqref{eq:starnare}, i.e.,
\[
X^\stella D^\stella +A^\stella X -X^\stella B^\stella X+C^\stella=0.
\]
Therefore, one has
\begin{equation}\label{eq:2}
	M^\stella\begin{bmatrix}
	I \\ X
	\end{bmatrix}
=	\begin{bmatrix}
	C^\stella & A^\stella\\ D^\stella & -B^\stella
	\end{bmatrix}
	\begin{bmatrix}
	I \\ X
	\end{bmatrix}
		=
   \begin{bmatrix}
	-X^\stella \\ I
	\end{bmatrix}(D^\stella-B^\stella X).
\end{equation}
From equations \eqref{eq:1} and \eqref{eq:2} we obtain 
\[
	\varphi(z)\begin{bmatrix} I\\X\end{bmatrix} = \begin{bmatrix} -X^\stella\\ I\end{bmatrix} (A-BX+z(D^\stella-B^\stella X)),
\]
i.e., the columns of $\left[\begin{smallmatrix} I\\X\end{smallmatrix}\right]$ span a deflating subspace of the
pencil $\phi(z)$, corresponding to the eigenvalues of the pencil $\alpha(z)$.

Conversely, assume that the columns of $\left[\begin{smallmatrix} I\\X\end{smallmatrix}\right]$ span a deflating subspace of $\phi(z)$, 
such that the corresponding eigenvalues form a reciprocal-free set.
  
 We may easily observe that $X$ is a solution to \eqref{eq:starnare} if and only if 
\[
\begin{bmatrix}
	I & X^\stella
	\end{bmatrix} M
	\begin{bmatrix}
	I\\ X 
	\end{bmatrix}=0,
\]
i.e., the matrix $\left[\begin{smallmatrix} I\\X\end{smallmatrix}\right]$ is $M$-isotropic. Since $\phi(z)=M+zM^\stella$ is a regular \mbox{$\stella$-palindromic} pencil, the matrix $\left[\begin{smallmatrix} I\\X\end{smallmatrix}\right]$ is $M$-isotropic thanks to~\cite[Thm.~3.3]{m4}, and thus
$X$ solves 
\eqref{eq:starnare}. 
\end{proof}

In the following, when 
$X$ is a solution to \eqref{eq:starnare} such that the columns of $\left[\begin{smallmatrix} I\\X\end{smallmatrix}\right]$ span a deflating subspace of the regular pencil $\varphi(z)$, associated with the eigenvalues 
$\lambda_1,\ldots,\lambda_n$, 
we say that $X$ is associated with the 
eigenvalues $\lambda_1,\ldots,\lambda_n$ of $\varphi(z)$ as well.
Notice that, in general, the solution $X$ is a complex matrix. If the complex eigenvalues in the set $\{
\lambda_1,\ldots,\lambda_n\}$ appear in complex conjugate pairs, then the corresponding deflating subspace of $\varphi(z)$ is spanned by the columns of a real matrix, therefore the
solution $X$ is real.


An interesting consequence of Theorem~\ref{thm:main} is that certain solutions of the $\stella$-NARE are solutions of a suitable Discrete-time Algebraic Riccati Equation (DARE).

\begin{corollary}\label{thm:dare}
If $X$ is a solution to \eqref{eq:starnare}
such that $\det(A-BX)\ne 0$, then $X$ is a solution to
\begin{equation}\label{eq:dare1}
C^\stella+A^\stella X=(C+DX)(A-BX)^{-1}(D^\stella -B^\stella X).
\end{equation}
Similarly, if $X$ is a solution to \eqref{eq:starnare}
such that  $\det(D^\stella-B^\stella X)\ne 0$ then
 $X$ is a solution to
\[
C+DX=(C^\stella+A^\stella X)(D^\stella-B^\stella X)^{-1}(A-BX).
\]
\end{corollary}
\begin{proof}
Since $X^\stella (A-BX)+C+DX=0$,
if $A-BX$ is invertible, then $X^\stella =-(C+DX)(A-BX)^{-1}$. By post-multiplying the latter equation by $D^\stella -B^\stella X$ and by using \eqref{eq:starnare}, one gets \eqref{eq:dare1}. The case where $\det(D^\stella -B^\stella X)\ne 0$  can be treated analogously.
\end{proof}

The following result gives sufficient conditions under which the assumptions of Corollary~\ref{thm:dare} are satisfied.

\begin{proposition}
Assume that the pencil $\varphi(z)$ in \eqref{eq:phi} is regular.
If $X$ is a solution to \eqref{eq:starnare}
associated with a reciprocal-free set of eigenvalues of $\varphi(z)$, then at least one between the matrices $A-BX$ and $D^\stella-B^\stella X$ is nonsingular.

If the matrix $M$ in \eqref{eq:phi} is nonsingular, then any solution $X$ to \eqref{eq:starnare} is such that $\det(A-BX)\ne 0$ and  $\det(D^\stella-B^\stella X)\ne 0$.
\end{proposition}

\begin{proof}
{}From Theorem~\ref{thm:main} the solution $X$ is associated with the eigenvalues 
of the pencil
$\alpha(z)$ in \eqref{eq:alpha}.
Since the eigenvalues of the pencil $\alpha(z)$ constitute a reciprocal-free set by hypothesis,
then $0$ and $\infty$ cannot be both eigenvalues, therefore one between $A-BX$ and $D^\stella-B^\stella X$ is nonsingular.
Concerning the second part, 
assume that $M$ is nonsingular, then $0$ and $\infty$ cannot be eigenvalues of $\varphi(z)$ and thus of $\alpha(z)$, this show that both $A-BX$ and $D^\stella-B^\stella X$ are nonsingular.
\end{proof}

Under the assumption that $\phi(z)$ has no eigenvalues on the unit circle, the solution $X$ associated with the eigenvalues of $\phi(z)$ lying inside (or outside) the unit circle is unique and real, according to the 
following result.

\begin{theorem}\label{thm:uniqe}
Let $\phi(z)$ and $\alpha(z)$ be as in \eqref{eq:phi} and \eqref{eq:alpha}, respectively. If $\phi(z)$ is regular with no eigenvalues on the unit circle and there exists an $n$-dimensional graph deflating subspace $\left[\begin{smallmatrix} I\\X\end{smallmatrix}\right]$ of $\phi(z)$ corresponding to eigenvalues lying inside the open (outside the closed) unit disk, then $X$ is the unique solution to \eqref{eq:starnare} such that the eigenvalues of $\alpha(z)$ are inside the open (outside the closed) unit disk.  Moreover, $X$ is a real matrix.
\end{theorem}

\begin{proof}
Since $\phi(z)$ is a palindromic pencil, its eigenvalues come in pairs $(\lambda,1/\lambda)$, and since $\det\phi(z)\ne 0$ if $|z|=1$, then $\phi(z)$ has $n$ eigenvalues inside the open unit disk and $n$ eigenvalues outside the closed unit disk. Since the $n$ eigenvalues inside the open (outside the closed) unit disk are disjoint from the remaining $n$, then the deflating subspace corresponding to the $n$ eigenvalues 
inside the open (outside the closed) unit disk  
 is unique \cite{bb06}. The thesis follows from Theorem~\ref{thm:main} and from the property that complex conjugate eigenvalues have  the same modulus.
\end{proof}

By borrowing an idea from the study of nonsymmetric algebraic Riccati equations \cite{bim:book}, we
consider the dual equation
\begin{equation}\label{eq:dual}
Y^\stella D+AY-B+Y^\stella C Y=0,
\end{equation}
that is obtained by exchanging the role of $A$ and $D$, and the role of $B$ and $C$ in \eqref{eq:starnare}.
By following the same arguments used in Theorem~\ref{thm:main} for equation \eqref{eq:starnare}, we find that, if $Y$ is a solution to \eqref{eq:dual}, then
\begin{equation}\label{eq:tinv}
\phi(z)\left[\begin{array}{c}
Y \\
I
\end{array}\right]=  \left[\begin{array}{c}
I \\
-Y^\stella
\end{array}\right]\beta(z) ,~~~\beta(z)=D+CY+z(A^\stella + C^\stella Y).
\end{equation}
We say that $Y$ is 
associated with the eigenvalues of $\beta(z)$.

Therefore, if $X$ and $Y$ are solutions to \eqref{eq:starnare} and \eqref{eq:dual}, respectively, then
 \[
\phi(z)\left[\begin{array}{cc}
I & Y \\
X & I
\end{array}\right]=
\left[\begin{array}{cc}
-X^\stella & I \\
I & -Y^\stella
\end{array}\right]
\left[\begin{array}{cc}
\alpha(z) & 0 \\
 0 & \beta(z)
\end{array}\right].
\]
Hence, if $X$ and $Y$ are such that $I-XY$ is invertible (that is equivalent to require that $\left[\begin{smallmatrix} I & Y\\X & I\end{smallmatrix}\right]$ is invertible), then we derive the following result, that provides a block diagonalization of the pencil $\phi(z)$.

\begin{theorem}
Let $X$ and $Y$ be solutions to equations \eqref{eq:starnare} and \eqref{eq:dual}, respectively, such that $\det(I-XY)\ne 0$. Then
\[
\phi(z)=
\left[\begin{array}{cc}
-X^\stella & I \\
I & -Y^\stella
\end{array}\right]
\left[\begin{array}{cc}
\alpha(z) & 0 \\
 0 & \beta(z)
\end{array}\right]
\left[\begin{array}{cc}
I & Y \\
X & I
\end{array}\right]^{-1},
\]
where $\alpha(z)$ and $\beta(z)$ are defined in \eqref{eq:alpha} and \eqref{eq:tinv}, respectively.
\end{theorem}


We conclude with a result that will be useful to show the convergence of doubling algorithms for computing the solution associated with the eigenvalues inside/outside the unit circle.

\begin{theorem}\label{thm:SDA0}
Assume that $\varphi(z)$ is regular, with $\det \varphi(z)\ne 0$ when $|z|=1$, and that $\phi(z)$ has two deflating subspaces spanned by the columns of $\left[\begin{smallmatrix} I\\X\end{smallmatrix}\right]$ and $\left[\begin{smallmatrix} Y\\I\end{smallmatrix}\right]$, that in turn are associated with the eigenvalues inside and outside the disk, respectively.  Then $X$ and $Y$ solve \eqref{eq:starnare} and \eqref{eq:dual}, respectively,
the matrices $D^T-B^TX$ and $D-CY$ are invertible, and 
\begin{equation}\label{eq:de}
	M\begin{bmatrix} I\\X\end{bmatrix}
	=M^T\begin{bmatrix} I\\X\end{bmatrix} W,~~
	M\begin{bmatrix} Y\\I\end{bmatrix}V= M^T\begin{bmatrix} Y\\I\end{bmatrix},
\end{equation}
where $W=(D^T-B^TX)^{-1}(A-BX)$, $V=(D+CY)^{-1}(A^T+C^TY)$. Moreover, $\rho(W)=\rho(V)<1$.

\end{theorem}
\begin{proof}
From the second part of Theorem~\ref{thm:main}, it follows that $X$ is a solution to~\eqref{eq:starnare}. Hence, from the first part of Theorem~\ref{thm:main}, equation \eqref{eq:defl} is verified.
 Under our assumptions,  the eigenvalues of $\alpha(z)$ lie inside the unit disk. This implies that $D^T-B^TX$ is invertible since, if it were singular, then $\alpha(z)$ in \eqref{eq:alpha} would have an eigenvalue at $\infty$. By equating the terms in $z$ in \eqref{eq:defl}, we obtain
\[
\begin{bmatrix} -X^T\\I \end{bmatrix}=
M^T \begin{bmatrix} I \\X \end{bmatrix}W
\]
where $W=(D^T-B^TX)^{-1}(A-BX)$. By plugging this expression in \eqref{eq:defl} and by comparing the constant terms, we obtain
\[
	M\begin{bmatrix} I\\X\end{bmatrix}
	=M^T\begin{bmatrix} I\\X\end{bmatrix} W.
\]
Since the eigenvalues of $W$ coincide with the eigenvalues of $\alpha(z)$, we have that {$\rho(W)<1$}. Similarly, by using \eqref{eq:tinv}, we proceed for the deflating subspace
$\left[\begin{smallmatrix} Y\\I\end{smallmatrix}\right]$.
The property $\rho(V)=\rho(W)$ holds since the eigenvalues of $V$ are the reciprocals of the eigenvalues of $\beta(z)$ in \eqref{eq:tinv}, that coincide with the eigenvalues of $\alpha(z)$.
\end{proof}

%
%
%
%
%
%
%

\section{A $T$-NARE with nonnegativity properties}\label{sec:posi}

Given a real matrix $A$, we write $A\ge 0$ ($A>0$) if all the entries of $A$ are nonnegative (positive), and we say that the matrix $A$ is nonnegative (positive). Moreover, we write $A\ge B$ if $A-B\ge 0$.

An usual assumption for the NARE \eqref{eq:nare}
is that the matrix
\begin{equation}\label{eq:hatM}
	\widehat M=\begin{bmatrix} A & -B\\ C & D\end{bmatrix},
\end{equation}
is an $M$-matrix, i.e., a matrix of the form $\widehat M=\sigma I-H$, where $H\ge 0$ and $\rho(H)\le \sigma$. Under mild assumptions on $\widehat M$, it can be proved that there exists a minimal solution of the NARE \eqref{eq:nare}. More precisely, if 
there exists a vector $v>0$ such that $\widehat Mv\ge 0$, then the NARE \eqref{eq:nare} has a minimal nonnegative solution \cite[Thm. 2]{guo12}, where the ordering is meant component-wise. This holds, in particular, when $\widehat M$ is nonsingular or singular irreducible.

Apparently, the case of the $\stella$-NARE \eqref{eq:starnare} where $\wh M$ is an $M$-matrix is much more complicated than the case of a NARE \eqref{eq:nare}, and has been treated in \cite{bp20}.
In particular, in \cite{bp20}, the existence of a minimal nonnegative solution is guaranteed under the hypothesis that the matrix 
\[
W=I\otimes D+(A^\stella \otimes I)\Pi
\] 
has a nonnegative inverse, where $\Pi=\sum_{i,j}e_ie_j^T\otimes e_ie_j^T$ ($e_k$ is the $k$-th vector of the canonical basis of $\mathbb{R}^n$)
and that there exists a matrix $\wt X$ such that
\begin{equation}\label{eq:seq0}
   	D\wt X+\wt X^\stella A-\wt X^\stella B\wt X+C\ge 0.
\end{equation}

We will show that under milder assumptions than those in \cite{bp20},
there exists a minimal nonnegative solution
$X_{\min}$, i.e., such that $X_{\min}\le X$ for any other possible nonnegative solution $X$.

For this purpose, define the sequence $\{ X_\ell\}_\ell$ by means of the recursion
\[
D X_{\ell+1}+X_{\ell+1}^TA=X_\ell^T B X_\ell-C,\quad\ell=0,1,\ldots,\quad X_0=0.
\]
Under suitable assumptions on the coefficients of the $\stella$-NARE, the sequence $\{ X_\ell\}_\ell$ is well defined and converges to $X_{\min}$, as stated in the following proposition.

\begin{proposition}
Assume that $\det (W)\ne 0$, $W^{-1}\ge 0$, $D^{-1}\ge 0$, $B\ge 0$, $C\le 0$, and that there exist  $u,v\in \mathbb{R}^n$ positive vectors such that
$Au-Bv\ge 0$ and $Cu+Dv\ge 0$.
Then there exists a minimal nonnegative solution $X_{\min}$ to \eqref{eq:starnare} and the sequence $\{X_\ell\}_\ell$ generated by \eqref{eq:seq0} is well defined and converges to $X_{\min}$, with $0\le X_\ell \le X_{\ell+1}$, and $X_\ell u\le v$ for any $\ell\ge 0$.
\end{proposition}

\begin{proof}
We show by induction on $\ell$ that $0\le X_\ell \le X_{\ell+1}$ and $X_\ell u\le v$. If $\ell=0$, then $X_0u=0\le v$ and $X_1$ solves the equation  $DX_{1}+X_{1}^TA=-C$, i.e.,
$\mathrm{vec}(X_1)=-W^{-1}\mathrm{vec}(C)$. Since $W^{-1}\ge 0$ and $C\le 0$, then $0=X_0\le X_1$. Under the assumption $0\le X_\ell \le X_{\ell+1}$, $X_\ell u\le v$, we show that $0\le X_{\ell+1}\le X_{\ell+2}$.
Since $W^{-1}\ge 0$, $B\ge 0$ and $C\le 0$, we have
\[
\mathrm{vec}(X_{\ell+2})=W^{-1}\mathrm{vec}(X_{\ell+1}^TBX_{\ell+1}-C)\ge W^{-1}\mathrm{vec}(X_{\ell}^TBX_{\ell}-C)=\mathrm{vec}(X_{\ell+1})\ge 0.
\]
Now we show that $X_{\ell+1}u\le v$.
The matrix $X_{\ell+1}$ is such that
\[
X_{\ell+1}=D^{-1}(X_\ell^TBX_\ell-C-X_{\ell+1}^TA).
\]
Multiplying to the right by $u$ yields
\begin{equation}\label{eq:yu}
X_{\ell+1}u=D^{-1}(X_\ell^TBX_\ell-C-X_{\ell+1}^TA)u.
\end{equation}
From the inequality $Au-Bv\ge 0$, since $D^{-1}\ge 0$ and $X_{\ell+1}\ge 0$, we deduce that $-D^{-1}X_{\ell+1}^TAu\le -D^{-1}X_{\ell+1}^TBv$. Moreover, from the inequality $Cu+Dv\ge 0$, since $D^{-1}\ge 0$, we deduce that $-D^{-1}Cu\le v$.
Hence, since $B\geq 0$ and $X_\ell u\le v$, from \eqref{eq:yu}, we derive
\[
X_{\ell+1}u\le v + D^{-1}(X_\ell^T-X_{\ell+1}^T)Bv\le v,
\]
where the latter inequality follows since $0\le X_\ell\le X_{\ell+1}$.
Therefore, the sequence $\{ X_\ell\}_\ell$ is monotonic nondecreasing and bounded from above, since $X_\ell\ge 0$ and $X_\ell u\le v$, with $u,v>0$. Therefore there exists $X=\lim_\ell X_\ell$.
Such a matrix $X$ solves \eqref{eq:starnare} by continuity and is a nonnegative solution. We show that it is the minimal nonnegative solution. Assume that $Y\ge 0$ is a solution to \eqref{eq:starnare}. We may easily show by induction on $\ell$ that $X_\ell\le Y$ for any $\ell\ge 0$. Therefore the inequality holds also in the limit and $X_{\min}:=X$ is the minimal nonnegative solution. 
\end{proof}

In the $M$-matrix case, it  might be useful to consider the pencil
\[
	\widehat \phi(z)=\begin{bmatrix} A & -B\\C & D\end{bmatrix}
	+ z 	\begin{bmatrix} D^\top & -B^\top\\C^\top & A^\top\end{bmatrix},
\]
that has an interesting structure. Indeed, when $\widehat M$ in~\eqref{eq:hatM} is an $M$-matrix, then the two matrices in the pencil are $M$-matrices as well. 
For the pencil $\widehat\phi(z)$  we may formulate a result analogue to Theorem~\ref{thm:main}, that relates 
any solution of the $\stella$-NARE with a graph deflating subspace of the pencil $\widehat\phi(z)$.

%
%
%
%
%
%

\section{Algorithms based on the linearization}\label{sec:algs}

The equivalence of computing a graph deflating subspace of a palindromic pencil and a solution of a $\stella$-NARE allows one to devise new algorithms for the solution of the latter. In particular, we consider three algorithms for our problem:
\begin{itemize}
\item The QZ algorithm \cite[Chapter 7]{gvl};
\item The DA algorithm \cite[Chapter 5]{bim:book};
\item The palindromic QZ algorithm \cite{m4,ksw09}.
\end{itemize}
The first two are general algorithms that provide certain invariant subspaces of pencils, while the third one is specialized to the structure of the problem.

\subsection{The QZ algorithm} 
Given $A',B'\in\mathbb C^{n\times n}$, the QZ algorithm provides two unitary matrices $Q,Z$ such that $A'':=Q^*A'Z$ and $B'':=Q^*B'Z$ are upper triangular, where the superscript ``$*$'' denotes complex conjugate transposition. The pencil $A'+zB'$ is similar to the pencil $A''+zB''$, so if the former is regular, then its eigenvalues can be read from the diagonal of the latter, as the solutions of the equation $(A'')_{ii}-z (B'')_{ii}=0$, when $(B'')_{ii}\ne 0$, or $\infty$ when $(B'')_{ii}=0$, for $i=1,\ldots,n$ (the case $(A'')_{ii}=(B'')_{ii}=0$ cannot hold for a regular pencil).

One can also get deflating subspaces from the matrix $Z$. Indeed, for $\ell=1,\ldots,n$, the first $\ell$ columns of $Z$ span a right deflating subspace of the pencil $A'+zB'$, corresponding to the eigenvalues related to the first $\ell$ diagonal entries of $A''+zB''$ from the top left corner.

There exists a variant of the QZ algorithm for real matrices $A'$ and $B'$ such that $Q$ and $Z$ are real orthogonal and $A''+zB''$ is an upper quasi-triangular pencil, that is $A''+zB''$ is block upper triangular with diagonal blocks of size $1$, corresponding to real eigenvalues, or $2$, corresponding to pairs of complex conjugate eigenvalues. Also in the real case, the first $\ell$ columns of $Z$ yield a real deflating subspace corresponding to the first $\ell$ eigenvalues related to the diagonal blocks of $A''+zB''$, with the only constraint on $\ell$ that a $2\times 2$ diagonal block cannot be split. In particular, if the sizes of the diagonal blocks are $\nu_1,\ldots,\nu_h$, then the first $\sum_{i=1}^{j} \nu_i$ columns of $Z$ provide a real deflating subspace for $j=1,\ldots,h$.

In our problem,  $\varphi(z)=M+zM^T$ is a real pencil of size $2n$. If $\det\phi(z)\ne 0$ on the unit circle, one may want a deflating subspace corresponding to the $n$ eigenvalues inside/outside the unit disk. This can be obtained by relying on the reordered real QZ algorithm that produces two upper quasi-triangular matrices $M':=Q^TMZ$ and $M'':=Q^TM^TZ$, such that the  $n$ eigenvalues of the leading $n\times n$ principal submatrix of  $M'+zM''$ are the ones inside the unit disk. The span of the first $n$ columns of $Z$ is the required deflating subspace.

In principle, one can compute a deflating subspace corresponding to any subset of the spectrum, provided this is reciprocal-free, to get a solution to~\eqref{eq:starnare}, in view of Theorem~\ref{thm:main}.
The general procedure is summarized in Algorithm~\ref{alg:QZ}.

More details on the QZ algorithm can be found in \cite[Chapter 7]{gvl}.

%
%
\begin{algorithm}[t]
  \DontPrintSemicolon{}
  \SetKwInOut{Input}{input}\SetKwInOut{Output}{output}
  \Input{$M\in\mathbb{R}^{2n\times 2n}$ as in~\eqref{eq:phi} such that $\phi(z)=M+zM^{\stella}$ is regular and has an $n$-dimensional graph deflating subspace corresponding to a set $\Lambda$ of $n$ reciprocal free eigenvalues of $\phi(z)$.}
  \Output{An approximation $X$ to the solution of~\eqref{eq:starnare}, corresponding to the eigenvalues in $\Lambda$.}
  \BlankLine{}
  Compute unitary matrices $Q$ and $Z$ such that $M'=Q^{*}MZ$ and $M''=Q^{*}M^{\stella}Z$ are upper triangular \;
  
  If needed, apply a reordering procedure on the diagonal of $M'+zM''$, getting $\wt Z$ such that the first $n$ columns of $\wt Z$ span the deflating subspace corresponding to the eigenvalues in the set $\Lambda$; otherwise, set $\wt Z=Z$\;
  
  If $\wt Z=\begin{bmatrix}
   \wt Z_{11} & \wt Z_{12}\\
   \wt Z_{21} & \wt Z_{22}\\                                                                                                   \end{bmatrix}$ is such that $\wt Z_{11}$ is nonsigular, set $X=\wt Z_{21}\wt Z_{11}^{-1}$ \;
  \caption{QZ algorithm to solve~\eqref{eq:starnare}.\label{alg:QZ}}
\end{algorithm}


\subsection{The Doubling Algorithm} Given a pencil $\phi(z)=A'+zB'$, where $A',B'\in\mathbb C^{(m+n)\times (m+n)}$, with $m$ eigenvalues inside the unit disk and $n$ outside, under suitable assumptions, the (Structured) Doubling Algorithm (DA) algorithm allows one to find a deflating subspace corresponding to the eigenvalues inside and outside the unit disk.

For a complete description of doubling algorithms, we refer the reader to \cite[Chapter 5]{bim:book} and \cite{hll}.

We restrict our attention to the case of interest, where $\phi(z)=M+zM^T$, with $M$ defined in \eqref{eq:phi}, thus $m=n$.
In particular we are interested in computing the matrices $X$ and $Y$ which define the deflating subspaces in equation \eqref{eq:de}, and which in turn solve equations \eqref{eq:starnare} and \eqref{eq:dual}, respectively.

First, one constructs a new pencil $N+zK$, right equivalent to $\varphi(z)$, but in the form
\begin{equation}\label{eq:form}
	N=\begin{bmatrix}
	E_0 & 0 \\ -P_0 & I
	\end{bmatrix},\qquad 
	K=\begin{bmatrix}
	I & -G_0\\ 0 & F_0
	\end{bmatrix}.
\end{equation}
This is possible if and only if the matrix $S=\left[\begin{smallmatrix} C^T & D\\  D^T & -B \end{smallmatrix}\right]$ is invertible. Moreover, the matrices $E_0$, $F_0$, $G_0$, $P_0$ can be recovered by comparing the block entries in equation \cite{bmf}
\begin{equation}\label{eq:pencil}
 N+zK=S^{-1}(M+zM^T).
\end{equation}

Then, the doubling algorithm consists in generating the sequences
\begin{equation}\label{eq:SDAiteration}
\begin{split}
	E_{\ell+1}=E_\ell(I-G_\ell P_\ell)^{-1} E_\ell,&\quad P_{\ell+1}=P_\ell+F_\ell(I-P_\ell G_\ell)^{-1} P_\ell E_\ell,\\
	F_{\ell+1}=F_\ell(I-P_\ell G_\ell)^{-1} F_\ell,&\quad G_{\ell+1}=G_\ell+E_\ell(I-G_\ell P_\ell)^{-1}G_\ell F_\ell,\\
	\end{split}
\end{equation}
for $\ell\geq0$, which are well defined if, at each step, $I-G_\ell P_\ell$ and $I-P_\ell G_\ell$ are invertible. The (at least) quadratic convergence of the sequences generated by the doubling algorithm is stated in the following results.
\begin{theorem}\label{thm:SDA}
Let $M$ be as in \eqref{eq:phi} such that the assumptions of Theorem~\ref{thm:SDA0} are satisfied.
 If the sequences
$\{E_\ell\}_\ell,\{F_\ell\}_\ell,\{G_\ell\}_\ell,\{P_\ell\}_\ell$ 
generated by the doubling iteration \eqref{eq:SDAiteration},
with $E_0,F_0,G_0,P_0$ as in \eqref{eq:form} and \eqref{eq:pencil}, are well defined, then for any matrix norm,
\[
\begin{split}
 \limsup_{\ell\to\infty}\sqrt[2^\ell]{\|P_\ell-X\|}\le \tau^2,\qquad&
 \limsup_{\ell\to\infty}\sqrt[2^\ell]{\|H_\ell-Y\|}\le \tau^2,\\
 \limsup_{\ell\to\infty}\sqrt[2^\ell]{\|E_\ell\|}\le \tau,\qquad&
 \limsup_{\ell\to\infty}\sqrt[2^\ell]{\|F_\ell\|}\le \tau,
\end{split}
\]
where $X$ and $Y$ are the solutions to \eqref{eq:starnare} and \eqref{eq:dual} corresponding to the eigenvalues of $\phi(z)$ inside and outside the unit circle, respectively, and
$\tau=\rho\bigl((D^T-B^TX)^{-1}(A-BX)\bigr)=\rho\bigl((D+CY)^{-1}(A^T+C^TY)\bigr)<1$.
\end{theorem}
\begin{proof}
Theorem~\ref{thm:SDA0} implies that \eqref{eq:de} holds, therefore the proof follows from the application of
Theorem 5.3 of \cite{bim:book}.
\end{proof}

The quadratic convergence of the sequence $\{P_\ell\}_{\ell}$ to the solution $X$ can be proved also with the assumption that $\{F_\ell\}_\ell$ is bounded and if the pencil $\varphi(z)$ satisfies the hypotheses of Theorem \ref{thm:uniqe} (see \cite[Thm. 5.3]{bim:book}).

We wish to point out that the application of the Doubling Algorithm here is different  from the case of the NARE \eqref{eq:nare}, where one looks for a graph invariant subspace of the matrix $H$ in \eqref{eq:H} corresponding to eigenvalues on the left/right half plane. Indeed, for solving a NARE with the doubling algorithm, a preprocessing step is needed to get a pencil whose graph invariant subspace of interest has eigenvalues inside the unit disk. This preprocessing consists in applying rational transformations to the pencil $H-z I$, like the (generalized) Cayley transform. The different rational transformations lead to the variants of the DA algorithm, namely
 SDA, SDA with shrink-and-shift, ADDA \cite{hll}, that have been treated differently in the literature. In our case, for the $\stella$-NARE, the eigenvalues are naturally split with respect to the unit circle, so that we can directly apply the Doubling Algorithm starting from the form \eqref{eq:form}.




The Doubling Algorithm for~\eqref{eq:starnare} is reported in Algorithm~\ref{alg:SDA}.

\begin{algorithm}[t]
  \DontPrintSemicolon{}
  \SetKwInOut{Input}{input}\SetKwInOut{Output}{output}
  \Input{$M\in\mathbb{R}^{2n\times 2n}$ as in~\eqref{eq:phi} such that $\phi(z)=M+zM^{\stella}$ is regular and has an $n$-dimensional graph deflating subspace corresponding to
   $n$ eigenvalues inside the open unit disk;\\
  $\ell_{\max}>0$ maximum number of iteration allowed;\\
  $\varepsilon>0$ tolerance.}
  \Output{an approximation $X$ to the solution of~\eqref{eq:starnare} corresponding to the eigenvalues inside the open unit disk.}
  \BlankLine{}
  Compute $N=\begin{bmatrix}
	E_0 & 0 \\ -P_0 & I
	\end{bmatrix}=S^{-1}M$, and 
	$K=\begin{bmatrix}
	I & -G_0\\ 0 & F_0
	\end{bmatrix}=S^{-1}M^{\stella}$ where $S=\begin{bmatrix}
	C^{\stella} & D\\
	D^{\stella} & -B\\
  \end{bmatrix}$ \;
  
  \For{$\ell=0,\ldots,\ell_{\max}$}{
  \While{$\min\{\|E_{\ell}\|_\infty,\|F_{\ell}\|_\infty\}> \varepsilon$}{  
$
\begin{array}{rl}
	E_{\ell+1}=E_\ell(I-G_\ell P_\ell)^{-1} E_\ell,&\quad P_{\ell+1}=P_\ell+F_\ell(I-P_\ell G_\ell)^{-1} P_\ell E_\ell,\\
	F_{\ell+1}=F_\ell(I-P_\ell G_\ell)^{-1} F_\ell,&\quad G_{\ell+1}=G_\ell+E_\ell(I-G_\ell P_\ell)^{-1}G_\ell F_\ell,
	\end{array}
$
} 
   }
   Set $X=P_{\ell}$ \label{finish_line} \;
  \caption{Doubling Algorithm to solve~\eqref{eq:starnare}.\label{alg:SDA}}
\end{algorithm}




\subsection{The ordered palindromic QZ algorithm}\label{sec:palQR}

 The algorithms described in the previous sections do not exploit the palindromic structure of the pencil $\phi(z)=M+zM^\stella$. 
 
 A structured version of the QZ algorithm for palindromic pencils, namely the Palindromic QZ (PQZ) algorithm, has been developed and studied in \cite{m4}, \cite{Schroeder2007}, \cite{ksw09}. The idea is to reduce $\phi(z)$ to a similar pencil $R+zR^\stella$, where $R$ is anti-triangular.
 

We adapt the PQZ algorithm to our problem, endowing it with an ordering procedure on the anti-triangular entries of the matrix $R$ that allows us to deduce the desired deflating subspace.

The first step of the PQZ is the computation of the anti-triangular Schur form of $M$, which exists in view of the following result.

\begin{theorem}[\mbox{\cite[Theorems 2.3--2.5]{m4}}]\label{th:spectrum_pencil}
Let $M\in\mathbb{C}^{2n\times 2n}$. Then there exists a unitary matrix $U\in\mathbb{C}^{2n\times 2n}$ such that
\begin{equation}
 \label{eq:antitriangularform}
 R=U^TMU=\begin{bmatrix}
          0 & \cdots & 0 & r_{1,2n} \\
          \vdots & \iddots & \iddots & \vdots\\
          0 & \iddots & & \vdots\\
          r_{2n,1} & \cdots&\cdots& r_{2n,2n}\\
         \end{bmatrix},
         \end{equation}
 is in anti-triangular form.
 
 Moreover, let the pencil $\phi(z)=M+zM^{\stella}$ be regular and suppose that the matrix $R=U^TMU$ is anti-triangular. Then, the spectrum of $\phi(z)$ is given by 
 \begin{equation}\label{eq:spectrum}
  \left\{\lambda_j\;:\;r_{j,2n-j+1}\cdot\lambda_j=-r_{2n-j+1,j},\quad j=1,\ldots,2n\right\}\subset \mathbb C\cup\{\infty\},
\end{equation}
%
where the list $(\lambda_1,\ldots,\lambda_{2n})$ is reciprocally ordered, namely $\lambda_j=1/\lambda_{2n-j+1}$ for all $j=1,\ldots,2n$.
\end{theorem}
%
%
%
\begin{corollary}\label{cor:reciprocalfreeset}
 With the assumptions in Theorem~\ref{th:spectrum_pencil}, 
 if $\phi(z)$ has only zero or two eigenvalues of modulus $1$, then the subset 
 \[
 \left\{\lambda_1,\ldots,\lambda_n\right\}
 \]
 of the spectrum~\eqref{eq:spectrum} of $\phi(z)$ is reciprocal-free.
\end{corollary}
\begin{proof}
 The result directly comes from the fact that the spectrum~\eqref{eq:spectrum} is reciprocally ordered. Moreover, if $\lambda$ is an eigenvalue of $\phi(z)$ with modulus 1 and multiplicity 2, then 
 it appears only once in the first half of the spectrum, namely $\left\{\lambda_1,\ldots,\lambda_n\right\}
 $.
\end{proof}

In the next theorem we show how to construct the solution $X$ to~\eqref{eq:starnare} by exploiting the anti-triangular Schur form~\eqref{eq:antitriangularform} of $M$. To this end, we consider the following $2\times 2$ block partition of the unitary matrix $U$,
\begin{equation}\label{eq:Upartition}
 U=\begin{bmatrix}
    U_{11} & U_{12}\\
    U_{21} & U_{22}\\
\end{bmatrix}, \quad U_{ij}\in\mathbb{C}^{n\times n},\;i,j=1,2.
\end{equation}
\begin{theorem}
Let $\phi(z)=M+zM^T$ with $M$ as in \eqref{eq:phi}. Assume that $\phi(z)$ is regular, with exactly zero or two eigenvalues of modulus~$1$. Let $U$ be a unitary matrix such that $R=U^TMU$ is anti-triangular and its $(1,1)$-block $U_{11}$ with the notation in~\eqref{eq:Upartition} is nonsingular. Then the matrix
$X=U_{21}U_{11}^{-1}$ is a solution to~\eqref{eq:starnare}.
\end{theorem}
\begin{proof}
 By construction, the columns of the matrix $\begin{bmatrix}
              U_{11}\\
              U_{21}\\
             \end{bmatrix}
$
span the deflating subspace of $\phi(z)$ corresponding to the first $n$ eigenvalues $\{\lambda_1,\ldots,\lambda_n\}$. Similarly, the columns of
\[
\begin{bmatrix}
              U_{11}\\
              U_{21}\\
             \end{bmatrix}U_{11}^{-1}=\begin{bmatrix}
              I\\
              U_{21}U_{11}^{-1}\\
             \end{bmatrix},
\]
span the same subspace.

Since $\{\lambda_1,\ldots,\lambda_n\}$ is a reciprocal-free set of eigenvalues as shown in Corollary~\ref{cor:reciprocalfreeset}, the matrix $X=U_{21}U_{11}^{-1}$ is a solution to~\eqref{eq:starnare} thanks to Theorem~\ref{thm:main}.
\end{proof}
The PQZ scheme is thus very similar to its unstructured counterpart. However,
from a computational point of view, the design of structure-preserving algorithms for the anti-triangular Schur form~\eqref{eq:antitriangularform} of $M$ is a rather challenging task.

Different schemes have been illustrated in the literature and in all the numerical examples we present in this paper we employ~\cite[Algorithm 3.5]{m4} equipped with the palindromic QR procedure~\cite{Schroeder2007}.  

We must mention that the solution constructed from the first $n$ columns of the matrix $U$ obtained from the PQZ algorithm is associated with a reciprocal-free set $\left\{\lambda_1,\ldots,\lambda_n\right\}$ of eigenvalues of the pencil $\varphi(z)$, determined by the lower $n$ antidiagonal entries of $R+zR^T$. Nevertheless, one might be interested in a solution associated with a different set of reciprocal-free eigenvalues of $\varphi(z)$, for instance the set of eigenvalues lying inside or outside the unit disk, if $\det(\varphi(z))\ne 0$ on the unit circle.
A remedy is to postprocess the pencil $R+zR^T$ by applying a reordering procedure that puts the desired eigenvalues on the lower part of the antidiagonal, so that the first columns of the modified matrix $U$ yield the desired solution. Taking inspiration from \cite[Section 4.1]{ksw09} we now illustrate the procedure that guarantees the computation of the solution $X$ to~\eqref{eq:starnare} associated with the $n$ eigenvalues which lie inside the unit disk when the latter set exists. In principle, the same procedure can be applied to get a solution associated with any given reciprocal-free set of eigenvalues of $\varphi(z)$.


Given the $2n$ eigenvalues $\{\lambda_1,\ldots,\lambda_{2n}\}$ ordered as in~\eqref{eq:spectrum}, we start by defining 
$$k_1=\max\{i~|~i=1,\ldots,2n,~|\lambda_i|<1\}-n.$$
Clearly, $0\le k_1\le n$.
If $k_1=0$, we do not need to reorder any column of $U$ and $X=U_{21}U_{11}^{-1}$ is the sought solution.


If $k_1>0$, we consider the following $2k_1\times 2 k_1$ submatrix $R^{(1)}$ of $R$ 
\[
R^{(1)}=[e_{n-k_1+1},\ldots,e_{k_1+n}]^{\stella}R[e_{2-k_1+1},\ldots,e_{k_1+n}],
\]
partitioned as follows
\[
R^{(1)}=
\begin{bmatrix}
              0 & 0 & R_{13}^{(1)}   \\
              0&  R_{22}^{(1)} & R_{23}^{(1)}\\
              R_{31}^{(1)} & R_{32}^{(1)} & R_{33}^{(1)}\\
             \end{bmatrix},
\]
where $R_{13}^{(1)}$, $R_{31}^{(1)}$, and $R_{33}^{(1)}$ are scalars, $R_{22}\in\mathbb{R}^{2(k_1-1)\times 2(k_1-1)}$ is anti-triangular, and $R_{23}^{(1)}\in\mathbb{R}^{2(k_1-1)\times 1}$ while $R_{32}^{(1)}\in\mathbb{R}^{1\times2(k_1-1)}$.

We need to compute a $\stella$-congruence transformation such that
\[
\begin{bmatrix}
 W & Z^{\stella} & 1\\
 Y^{\stella} & I_{k_1-2} &0 \\
 1 & 0 &0\\
\end{bmatrix}\begin{bmatrix}
              0 & 0 & R_{13}^{(1)}   \\
              0&  R_{22}^{(1)} & R_{23}^{(1)}\\
              R_{31}^{(1)} & R_{32}^{(1)} & R_{33}^{(1)}\\
             \end{bmatrix}\begin{bmatrix}
 W & Y & 1\\
 Z & I_{k_1-2} & 0 \\
 1 & 0 &0 \\
\end{bmatrix}=
\begin{bmatrix}
              0 & 0 & R_{31}^{(1)}   \\
              0&  R_{22}^{(1)} & 0\\
              R_{13}^{(1)} & 0 & 0\\
             \end{bmatrix},
\]
where $W$ is a scalar, $Z\in\mathbb{R}^{2(k_1-1)\times 1}$, and $Y\in\mathbb{R}^{1\times 2(k_1-1)}$.

A direct computation shows that the vectors $Y$ and $Z$ can be computed by solving the following system of $\stella$-Sylvester equations
\begin{equation}\label{eq:TSylvester}
\left\{\begin{array}{rll}
         R_{31}^{(1)}Y+Z^{\stella}R_{22}^{(1)}&=&-R_{32}^{(1)},\\
         R_{13}^{(1)}Y+Z^{\stella}(R_{22}^{(1)})^{\stella}&=&-(R_{23}^{(1)})^{\stella}.
       \end{array}\right.
\end{equation}
The numerical solution of~\eqref{eq:TSylvester} can be carried out by, e.g., \cite[Algorithm 2]{dipr19}. In our numerical experiments, since the matrix variable are just vectors, we solve the linear system of equations arising by Kronecker transformations, namely 
\[
\begin{bmatrix}
 R_{31}^{(1)}\cdot I_{2(k_1-1)} & (R_{22}^{(1)})^{\stella}\\
 R_{13}^{(1)}\cdot I_{2(k_1-1)} & R_{22}^{(1)}\\
\end{bmatrix}
\begin{bmatrix}
 \text{vec}(Y)\\
 \text{vec}(Z^{\stella})\\
\end{bmatrix}
=
-\begin{bmatrix}
 \text{vec}(R_{32}^{(1)})\\
 \text{vec}((R_{23}^{(1)})^{\stella})\\
\end{bmatrix}.
\]
Once $Y$ and $Z$ are computed, we define
\[
W= -(R_{33}^{(1)}+R_{32}^{(1)}Z+Z^{\stella}R_{23}^{(1)}+Z^{\stella}R_{22}^{(1)}Z)/(R_{31}^{(1)}+R_{13}^{(1)}),
\]
and perform a QR factorization 
\[
\begin{bmatrix}
 W & Y & 1 \\
 Z & I_{2(k_1-1)} & 0 \\
 1 & 0 & 0 \\
\end{bmatrix}=P^{(1)}G.
\]
Then the first $n$ columns of the matrix 

\[
U^{(1)}=\begin{bmatrix}
 U^{(1)}_{11} & U^{(1)}_{12} \\
 U^{(1)}_{21} & U^{(1)}_{22} \\
\end{bmatrix}=
U\begin{bmatrix}
  I_{n-k_1} & 0 &0 \\
  0& P^{(1)} & 0\\
  0&0 & I_{n-k_1} \\
 \end{bmatrix},
\]
span a deflating subspace related to the set of eigenvalues $\{\lambda_i,\;i=1,\ldots,n,\,i\neq n-k_1+1\}\cup\{\lambda_{k_1+n}\}$.

We proceed by considering
\[
k_2=\max\{i~|~i=1,\ldots,k_1+n-1,~|\lambda_i|<1\}-n,
\]
and we repeat the same exact steps as before by defining $R=(U^{(1)})^{\stella}MU^{(1)}$. 

This procedure is iterated until 
\[
k_j=\max\{i~|~i=1,\ldots,
k_{j-1}+n-1,~|\lambda_i|<1\}-n.
\]
is equal to 0,
for a certain $j\geq 1$.

In conclusion, we construct the solution $X= U^{(j)}_{21}(U_{11}^{(j)})^{-1}$ which is related to the desired set of eigenvalues.

The overall palindromic QZ procedure is summarized in Algorithm~\ref{alg:palindromicQZ}.

\begin{algorithm}[t]
  \DontPrintSemicolon{}
  \SetKwInOut{Input}{input}\SetKwInOut{Output}{output}
  \Input{$M\in\mathbb{R}^{2n\times 2n}$ as in \eqref{eq:phi} and such that $\phi(z)=M+zM^{\stella}$ is regular.}
  \Output{$X$ approximate solution to~\eqref{eq:starnare}.}
  \BlankLine{}
  Compute the Schur anti-triangular form of $M$, namely the matrices $U$ and $R$ in~\eqref{eq:antitriangularform}, by~\cite[Algorithm 3.5]{m4}\;
  
  If needed, apply a reordering procedure on the diagonal of $R$, getting $\wt U$ such that its first $n$ columns span the deflating subspace corresponding to the eigenvalues inside/outside the unit disk; otherwise, set $\wt U=U$\;
  
  If $\wt U=\begin{bmatrix}
   \wt U_{11} & \wt U_{12}\\
   \wt U_{21} & \wt U_{22}\\                                                                                                   \end{bmatrix}$ is such that $\wt U_{11}$ is nonsingular,
  set $X=\wt U_{21}\wt U_{11}^{-1}$   \;
  \caption{PQZ algorithm to solve~\eqref{eq:starnare}.\label{alg:palindromicQZ}}
\end{algorithm}

%
%

\section{Numerical examples}\label{sec:test}
In this section we illustrate the numerical behavior of the proposed algorithms for  nonsymmetric $\stella$-NAREs~\eqref{eq:starnare}: the invariant subspace methods presented in the previous sections, namely the Doubling Algorithm, and the standard and palindromic QZ methods. 
We study the properties of the computed solutions and their relation to the spectrum of $\phi(z)$ as outlined in Section~\ref{sec:line} and compare them to the solutions computed by the Newton method for small-scale $\stella$-NAREs proposed in~\cite{bp20}. Unless stated otherwise, in all the following examples we always use the value $\varepsilon=10^{-12}$ in the convergence checks of both the Newton method and DA.

Results were obtained by running Matlab R2017b on a laptop with an Intel Core i3
processor running at 2GHz using 3.5GB of RAM.

The implementation of QZ is based on the Matlab routine {\tt ordqz}, the implementation of DA follows the code enclosed to \cite{bim:book}, while the implementation of PQZ is based on~\cite{m4}.  We note that
\cite[Algorithm 3.5]{m4} needs an input parameter $\alpha$; see \cite{m4} for more details. In all the examples reported in Section~\ref{sec:test} we use $\alpha=1.1$. From our numerical experience, the performance of the solver turned out to be quite robust under small changes of the parameter~$\alpha$.


\begin{num_example}\label{Ex.1}
 {\rm
 We consider the equation in Example 4.1 of \cite{bp20}, where the coefficient matrices defining the $T$-NARE~\eqref{eq:starnare} are
\[
A=\begin{bmatrix}
 -1 & -1 & & \\
   & -1  & -1 \\
   & & \hspace{-.6cm}\ddots&\hspace{-.6cm}\ddots \\
   &   &    -1  &  -1 \\
   & & & -1 \\
\end{bmatrix},\;
D=\begin{bmatrix}
 4 & -1 & & \\
   & 4  & -1 \\
   & & \hspace{-.6cm}\ddots&\hspace{-.6cm}\ddots \\
   &   &    4  &  -1 \\
   & & & 4 \\
\end{bmatrix},\; E=\begin{bmatrix}
 -1 & -1 & & \\
   & -1  & -1 \\
   & & \hspace{-.6cm}\ddots&\hspace{-.6cm}\ddots \\
   &   &    -1  &  -1 \\
   & & & -0.9 \\
\end{bmatrix},
\]
together with $B=-A/\|A\|_F$, and $C=E/\|E\|_F$.

As shown in~\cite{bp20}, there exists a nonnegative minimal solution $X_{\min}$ to this $T$-NARE and the Newton method with $X_0=0$ converges to $X_{\min}$.
 
 In Table~\ref{Ex1_table1} we report the relative residual norm 
 \[\|R\|_F=\frac{\|DX+X^{\stella}A-X^{\stella}BX+C\|_F}{\|X\|_F},
 \]
achieved by the numerical solution $X$ computed by DA, QZ, PQZ and the Newton scheme along with the relative distance between the Newton solution $X_{New}$ and the solution $X$ computed by the invariant subspace methods, namely $\| X_{New}-X\|_F/ \| X_{New}\|_F$, for different values of $n$.
 
 
 
 
 
 \setlength{\tabcolsep}{4pt}
  \begin{table}[t]
  \centering
   \begin{tabular}{c|c c cccc}
  & \multicolumn{2}{c}{$n=100$}
   & \multicolumn{2}{c}{$n=300$} & \multicolumn{2}{c}{$n=500$}\\
   & Rel. Res. & Rel. Dist. & Rel. Res. & Rel. Dist. &Rel. Res. & Rel. Dist. \\
   \hline
 
    PQZ&  3.11e-13 & 7.69e-13& 1.60e-12&5.51e-14 & 3.67e-12 & 9.49e-14\\
   QZ&1.70e-13&7.74e-13& 1.01e-12&3.88e-14&2.25e-12&6.59e-14\\
   
   DA& 8.64e-16&7.72e-13&6.36e-16&7.74e-15&7.76e-16&1.22e-14\\
   Newton&1.60e-12 & -- & 1.29e-13 & -- &  2.24e-13 &--\\
%
%

\end{tabular}
\caption{Example~\ref{Ex.1}. ``Rel. Res.'' denotes the relative residual norm achieved by PQZ, QZ, DA and Newton's scheme from~\cite{bp20}; ``Rel. Dist.'' the relative distance between the solutions computed by the invariant subspace methods and the one obtained by the Newton scheme. } \label{Ex1_table1}
\end{table}

From the results in Table~\ref{Ex1_table1} we can notice that the solutions computed by the methods we tested are very closed to each other. This means that for this example, also the invariant subspace methods are able to compute a very accurate approximation to the actual minimal nonnegative solution.

\begin{figure}[t]
  \centering
  \includegraphics[scale=0.7]{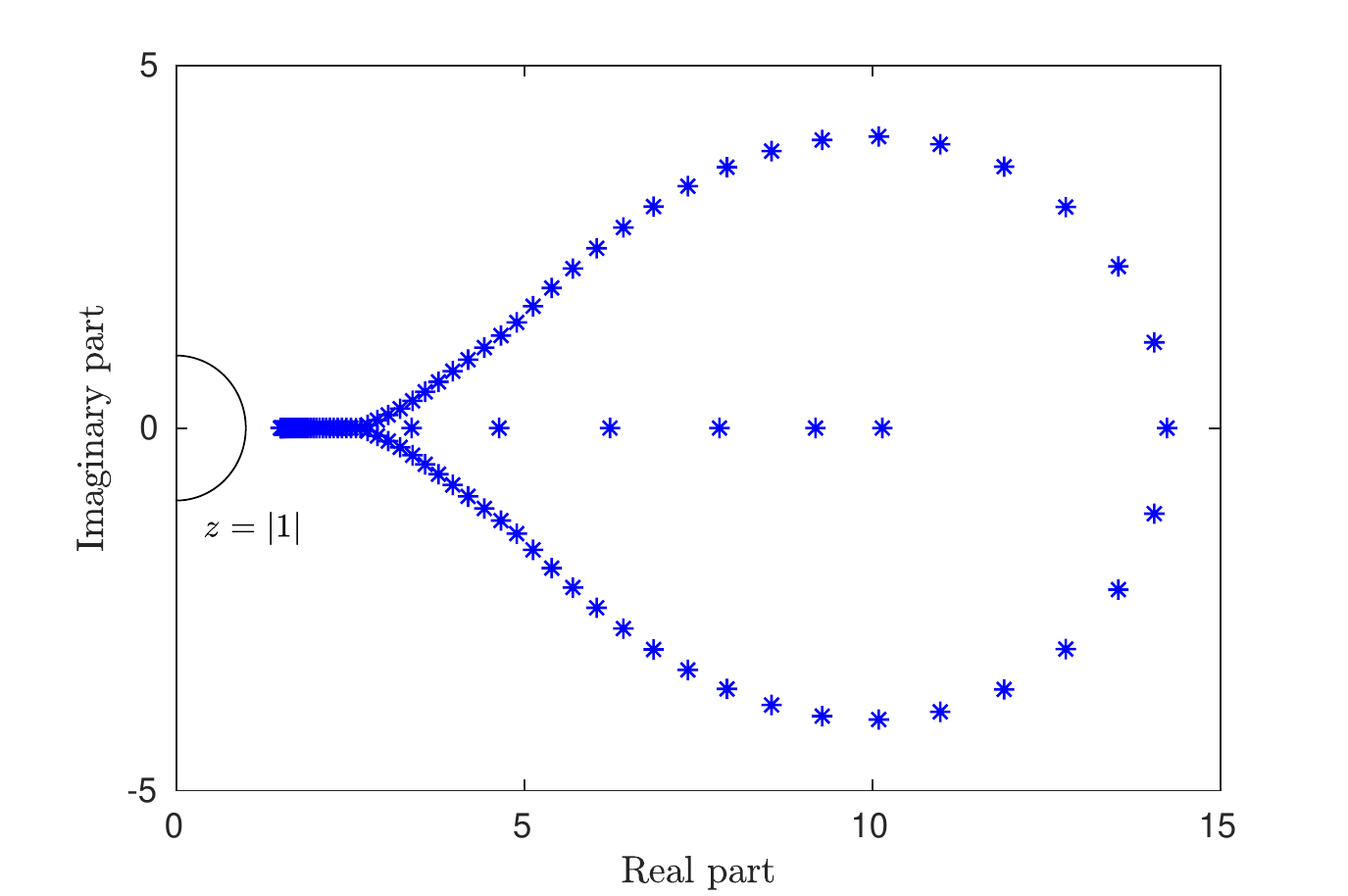}
\caption{Example~\ref{Ex.1}. Eigenvalue distribution of the pencil $\alpha(z)=A-BX+z(D^\stella-B^TX)$ where $X$ is the solution computed by the PQZ method for $n=100$. The black, solid line indicates (a portion of) the boundary of the unit disk.} \label{Ex1_fig1}
 \end{figure}

 In Figure~\ref{Ex1_fig1} we report the eigenvalue distribution of the matrix pencil $\alpha(z)=A-BX+z(D^\stella-B^TX)$ for $n=100$, where $X$ is the solution computed by the PQZ method, and we can see that the spectrum of $\alpha(z)$ is outside the unit circle. This happens for all the values of $n$ we tested. Therefore, for this example, the solution we compute can be characterized in two different ways. Indeed, it can be viewed as the unique minimal nonnegative solution as shown in~\cite{bp20}. On the other hand, Theorem~\ref{thm:main} says that it is also the unique solution such that the eigenvalues of $\alpha(z)$ are all outside the unit circle. 

To conclude, in Table~\ref{Ex1_table2} we report the running time of the PQZ method, QZ, DA, and the Newton procedure. We notice that the DA algorithm turns out to be the fastest method we tested. Even though this scheme needs more iterations than the Newton method to converge, each of these iterations is quite cheap as it involves only inversion and linear combinations of matrices of size $n$. On the other hand, the solution of a $\stella$-Sylvester equation needed at each Newton step is rather expensive increasing the computational cost of the overall scheme. 
The most demanding step of the PQZ and QZ methods is
the 
computation of the unitary matrices
$Q$ and $Z$ for the pencil $\varphi(z)$.
However,
for this example, both the latter algorithms 
are slightly faster than the Newton scheme.

   \begin{table}[t]
  \centering
   \begin{tabular}{ c| c c c c}
  $n$      & PQZ &QZ & DA &Newton \\
   \hline
   100 & 1.61e-1 & 
    1.67e-1& 3.19e-2 (7)&1.25e0 (3)\\

    300 &  3.44e0&4.28e0 &2.20e-1 (7)& 6.25e0 (3) \\

500& 1.87e1&2.28e1 &2.34e0 (7)& 2.00e1 (3)
\\

\end{tabular}
\caption{Example~\ref{Ex.1}. Computational time in seconds achieved by PQZ, QZ, DA, and Newton's scheme from~\cite{bp20}. For DA and Newton's method we report in brackets also the number of iterations needed to achieve the prescribed accuracy.} \label{Ex1_table2}
\end{table}

 }
\end{num_example}

\begin{num_example}\label{Ex.2}
 {\rm
 We consider Example~4.2 in \cite{bp20} where 
 the matrices $A,D\in\mathbb{R}^{n\times n}$ come from the finite difference discretization on the unit square of 2-dimensional differential operators
 equipped with homogeneous Dirichlet boundary conditions, while $B,C\in\mathbb{R}^{n\times n}$ are full random matrices.

 As before, we solve the related $\stella$-NARE by means of the PQZ method, QZ, DA, and the Newton method and in Table~\ref{Ex2_table1} we report the results.
 
 We wish to point out that, for this example, the properties of the coefficient matrices do not guarantee the existence of a unique minimal nonnegative solution. Nevertheless, in~\cite{bp20} it is shown that the Newton method is able to compute an accurate numerical solution in terms of relative residual norm. 
 
 \begin{table}[t]
  \centering
   \begin{tabular}{ c|c c cc}
  & \multicolumn{2}{c}{$n=324$}
   & \multicolumn{2}{c}{$n=784$} \\
   & Rel. Res. & Rel. Dist. & Rel. Res. & Rel. Dist. \\
   \hline
 
    PQZ&  7.01e-13 & 5.35e-14& 1.00e-12&4.52e-14 \\
   QZ&6.41e-13&5.23e-14& 1.14e-12&4.35e-14\\
   
   DA& 7.51e-14&4.95e-14&2.71e-11&4.87e-14\\
   Newton&1.67e-12 & -- & 7.98e-13 & -- \\
\end{tabular}
\caption{Example~\ref{Ex.2}. ``Rel. Res.'' denotes the relative residual norm achieved by PQZ, QZ, DA, and Newton's scheme from~\cite{bp20}; ``Rel. Dist.'' the relative distance between the solutions computed by the invariant subspace methods and the one obtained by Newton's scheme.} \label{Ex2_table1}
\end{table}

 From the results in Table~\ref{Ex2_table1} we can notice that also for this example the invariant subspace methods and  Newton's scheme basically compute the same numerical solution. In particular, for both the values of $n$ we tested, this solution is the unique solution such that the spectrum of $\alpha(z)$ is outside the unit circle. Therefore, even though
 we cannot characterize the computed solution in terms of its minimality, its uniqueness is guaranteed by Theorem~\ref{thm:main}. 
 
%
%
%
%
%
 
We compare 
the PQZ method, QZ, DA, and Newton's procedure also from a computational perspective and in Table~\ref{Ex2_table2} we report the results. 
Conclusions similar to the ones illustrated in Example~\ref{Ex.1} can be drawn.
The DA algorithm is still the fastest method we tested while  Newton's method suffers whenever a sizable number of iterations is needed to converge.
The PQZ and QZ methods perform better than Newton's method but are one order of magnitude slower than DA.

 \begin{table}[t]
  \centering
   \begin{tabular}{ c| c c c c}
  $n$      & PQZ &QZ & DA &Newton \\
   \hline
   324 & 3.62e0 & 
    4.51e0& 2.39e-1 (8)&1.66e1 (5)\\

    784 &  6.94e1&9.01e1 &3.11e0 (10)& 1.75e2 (8) \\

\end{tabular}
\caption{Example~\ref{Ex.2}. Computational time in seconds achieved by PQZ, QZ, DA, and Newton's scheme from~\cite{bp20}. For DA and Newton's method we report in brackets also the number of iterations needed to achieve the prescribed accuracy.} \label{Ex2_table2}
\end{table}
}
 
 \end{num_example}

\begin{num_example}\label{Ex.3}
{\rm
In this example we compare only the results achieved by the PQZ method and Newton's scheme. We consider the following $2\times 2$ problem
\[
\begin{array}{cccc}
D=\begin{bmatrix}
 1 & 0\\
 -0.1 & 2\\
\end{bmatrix}, & 
A=\begin{bmatrix}
 1 & -0.2\\
 -0.1 & 2\\
\end{bmatrix},&
B=\begin{bmatrix}
 0.2 & 0.1\\
 0.3 & 0.4\\
\end{bmatrix}, &
C=\begin{bmatrix}
 -0.1 & -0.1\\
 -0.1 & -0.1\\
\end{bmatrix}\\
  \end{array}.
\]
PQZ with no reordering computes the following solution 
\[
X_{out}\approx\begin{bmatrix}
         20.1028 & -25.4499 \\
         -11.5037 & 14.6980\\
          \end{bmatrix},
\]
which is the unique solution such that the spectrum of 
\[
\alpha_{out}(z)=A-BX_{out}+z(D^\stella-B^\stella X_{out}),
\]
is given approximately by $\{-1.09484,  -1.05880\}$ and lies outside the unit circle. Thanks to the strategy presented at the end of Section~\ref{sec:palQR}, we are also able to construct the matrix 
\[
X_{in}\approx\begin{bmatrix}
         2.6923 &   3.6756 \\
   1.9569 &  2.6749 \\  
         \end{bmatrix},
\]
which approximates the unique solution such that the eigenvalues of $\alpha_{in}(z)=A-BX_{in}+z(D^\stella-B^\stella X_{in})$ are inside the unit circle. The spectrum of $\alpha_{in}$ is approximately given by $\{-0.91338,
  -0.94447\}$. Both $X_{out}$ and $X_{in}$ achieve a relative residual norm of the order $\mathcal{O}(10^{-12})$.
  
  It is interesting to notice that Newton's method applied to this example computes neither $X_{out}$ nor $X_{in}$. The solution 
  \[
  X_{New}\approx\begin{bmatrix}
                 0.0490 &    0.1541\\
   -0.0220  &  0.0385\\
              \end{bmatrix},
\]
computed by Newton's method with zero initial guess is such that the spectrum of $\alpha_{New}(z)$ is approximately given by $\{-0.94447,-1.09484\}$. The latter is a reciprocal-free set but does
not lie inside/outside the unit circle. This example shows that the zero initial value does not implies the convergence of Newton's method to the unique solution $X$ related to a pencil $\alpha(z)$ whose eigenvalues are all inside/outside the unit circle. The choice of a better suited initial value is a topic that may be worth exploring and it will be the subject of future work.
} 
\end{num_example}

\begin{num_example}\label{Ex.4}
{\rm
In the last example we compare the accuracy in terms of forward error achieved by PQZ, QZ and DA with respect to a reference solution computed with high precision arithmetic.

Even though the numerical results reported in the previous examples show that the DA algorithm is very competitive in terms of running time, this method
does not exploit the palindromic structure of the problem with a possible loss of accuracy. Also the (plain) QZ method shares this drawback.

For a given $n$, we start by defining the matrix $\wt M\in\mathbb{R}^{2n\times 2n}$ entry-wise as follows


\[
 \wt M\;:\;\left\{\begin{array}{llll}
   \wt M_{i,2n-i+1}&=&i+1,& i=1,\ldots,n-1,\\
   \wt M_{2n-i+1,i}&=&1/(i+1),& i=1,\ldots,n-1,\\
   \wt M_{n+1,n}&=&\sigma+1,&\\
   \wt M_{n,n+1}&=&1/(\sigma+1),&\\
   \wt M_{i,j}&=&0,&i<j,\\  
   \wt M_{i,j}&=&1/5,& i>j,\\
  \end{array}\right.
\]
where $\sigma>0$ is given. 


Then the matrix $M$ in~\eqref{eq:phi} is defined as $M = N\wt MN^{\stella}$ where 
$$N=\begin{bmatrix}
   1 & 1 & \cdots & 1\\
   -1 & 1 &\cdots & 1\\
   \vdots& \ddots& \ddots& \vdots\\
   -1 & \cdots & -1& 1\\
  \end{bmatrix}.
$$
Finally, we partition $M$ as in~\eqref{eq:phi} and choose the coefficient matrices $A$, $D$, $B$, and $C$ defining equation~\eqref{eq:starnare} accordingly. 

Observe that the eigenvalues of the pencil $\phi(z)$ are
$\lambda_i,\lambda_i^{-1}$, $i=1,\ldots,n$, where $\lambda_1=1/(1+\sigma)^2$, and $\lambda_i=1/i^2$ for $i=2,\ldots,n$. Therefore, the smaller $\sigma$, the narrower the separation of the eigenvalues with respect to the unit circle. 

We choose $\sigma=10^{-10}$ and we construct the matrix $M$ 
by making use of the Matlab function {\tt vpa} which allows us to use variable precision arithmetic. In particular, we employ 100 decimal digit accuracy. 

We first solve the obtained equation by DA equipped with {\tt vpa}. The computed solution is considered as reference solution. Notice that the DA algorithm is very well-suited for the use of {\tt vpa} as it only involves linear combinations and inversions of matrices. 

We then solve the same equation by PQZ, QZ, and DA, all equipped with double precision arithmetic.

In Table~\ref{Ex4_table1} we report the relative forward error achieved by the aforementioned schemes, namely
$\| X_{Exact}-X\|_F/ \| X_{Exact}\|_F$, where $X_{Exact}$ is the solution computed by using the \texttt{vpa} arithmetic with 100 significant digits and rounded to double precision.

\begin{table}[t]
  \centering
   \begin{tabular}{ r| r r r }
  $n$      & PQZ &QZ & DA  \\
   \hline
   3 & 2.72e-16 & 
    9.29e-7& 7.41e-6\\

    4 &  4.95e-15&4.02e-5 &1.02e-5 \\

\end{tabular}
\caption{Example~\ref{Ex.4}. Relative forward error achieved by PQZ, QZ, and DA. The reference solution is computed by DA equipped by {\tt vpa}.} \label{Ex4_table1}
\end{table}
We can notice that, for this example, the PQZ method is able to achieve a very small relative forward error by solely relying on the full exploitation of the palindromic structure of the problem. On the other hand, QZ and DA completely neglect this property and they end up constructing less accurate solutions.

} 
\end{num_example}

%
%
%
%
%
%

\section{Conclusions}
The main result of this paper relates the solution of $\stella$-NAREs to deflating subspaces of a specific matrix pencil.
Such a novel relation allowed us to design methods to solve 
$\stella$-NAREs based on invariant subspace algorithms. In particular, the PQZ, QZ and DA schemes have been proposed and largely tested. We showed how DA is the fastest algorithm, while the PQZ scheme is able to achieve better accuracy in terms of forward error thanks to its fully exploitation of the palindromic structure of the underlying eigenvalue problem. 

\section*{Acknowledgements}
We thank C. Mehl for providing us with some Matlab code implementing the structured deflation scheme for palindromic pencils presented in~\cite{m4} to compute the anti-triangular Schur form of $M$.

 The last three authors are members of the Italian
research group Indam-GNCS, whose support is gratefully acknowledged. 

Part of this work was
carried out while the fourth author was affiliated with the Max Planck Institute for Dynamics of
Complex Technical Systems in Magdeburg, Germany.


\end{document}